\documentclass{amsart}
\usepackage{amssymb,amsmath,amscd,amsthm}
\usepackage{amsfonts,epsfig,latexsym,graphicx,amssymb,pict2e}

\newtheorem{Lm}{Lemma}[section]
\newtheorem{Thm}[Lm]{Theorem}
\newtheorem{Prop}[Lm]{Proposition}
\newtheorem{Cl}[Lm]{Claim}
\newtheorem{Rem}[Lm]{Remark}
\newtheorem{Cor}{Corollary}

\def\R{\mathbb{R}}

\def\Z{\mathbb{Z}}

\def\T{\mathbb{T}}

\def\bdef{\begin{Def}}
\def\endef{\end{Def}}
\def\bthm{\begin{Thm}}
\def\ethm{\end{Thm}}
\def\bprop{\begin{Prop}}
\def\enprop{\end{Prop}}
\def\blm{\begin{Lm}}
\def\elm{\end{Lm}}
\def\bcor{\begin{Cor}}
\def\ecor{\end{Cor}}
\def\brm{\begin{Rem}}
\def\erm{\end{Rem}}
\def\bfig{\begin{picture}}
\def\efig{\end{picture}}
\def\beq{\begin{eqnarray}}
\def\eneq{\end{eqnarray}}
\def\beal{\begin{aligned}}
\def\enal{\end{aligned}}

\title[The Density of States Measure of the Fibonacci Hamiltonian]{The Density of States Measure of the Weakly Coupled Fibonacci Hamiltonian}

\author{David Damanik}

\address{Department of Mathematics, Rice University, Houston, TX~77005, USA}

\email{damanik@rice.edu}

\thanks{D.\ D.\ was supported in part by NSF grants DMS--0800100 and DMS--1067988.}

\author{Anton Gorodetski}

\address{Department of Mathematics, University of California, Irvine CA 92697, USA}

\email{asgor@math.uci.edu}

\thanks{A.\ G.\ was supported in part by NSF grants DMS--0901627 and IIS-1018433.}

\date{\today}

\begin{document}

\begin{abstract}
We consider the density of states measure of the Fibonacci Hamiltonian and show that, for small values of the coupling constant $V$,  this measure is exact-dimensional and the almost everywhere value $d_V$ of the local scaling exponent is a smooth function of $V$, is strictly smaller than the Hausdorff dimension of the spectrum, and converges to one as $V$ tends to zero. The proof relies on a new connection between the density of states measure of the Fibonacci Hamiltonian and the measure of maximal entropy for the Fibonacci trace map on the non-wandering set in the $V$-dependent invariant surface. This allows us to make a connection between the spectral problem at hand and the dimension theory of dynamical systems.
\end{abstract}

\maketitle

%\tableofcontents

\section{Introduction}

The Fibonacci Hamiltonian is a central model in the study of electronic properties of one-dimensional quasicrystals. It is given by the following bounded self-adjoint operator in
$\ell^2(\Z)$,
$$
[H_{V,\omega} \psi] (n) = \psi(n+1) + \psi(n-1) + V
\chi_{[1-\alpha , 1)}(n\alpha + \omega \!\!\! \mod 1) \psi(n),
$$
where $V > 0$, $\alpha = \frac{\sqrt{5}-1}{2}$, and $\omega \in \T = \R / \Z$. It is well known and easy to see that the spectrum of $H_{V,\omega}$ does not depend on $\omega$ and hence may be denoted by $\Sigma_V$.

The Fibonacci Hamiltonian has been studied in numerous papers and many results have been obtained for it; compare the survey articles \cite{D00, D07, D09, S95}. For example, for every $V>0$, the spectrum $\Sigma_V$ is a Cantor set of zero Lebesgue measure and, for every $\omega$, the operator $H_{V,\omega}$ has purely singular continuous spectrum. Naturally, one wants to go beyond these statements and study the fractal dimension of the spectrum and of the spectral measures, as well as the transport exponents. Such finer issues are reasonably well understood for $V$ sufficiently large \cite{DEGT, DT03, DT07, DT08, KKL, Ra}. On the other hand, the methods used in the large coupling regime clearly do not extend to small values of $V$. For this reason, we have developed in a series of papers, \cite{DG1, DG2, DG3}, tools that allow one to study and answer the issues above in the small coupling regime. The results obtained in these papers are satisfactory as far as the spectrum as a set is concerned. While there are results for spectral measures and transport exponents, they do not seem to be optimal. In particular, in the $V \downarrow 0$ limit, the estimates we have obtained for these quantities do not approach the values the quantities take for $V = 0$. In this paper we introduce new ideas that lead to satisfactory results for phase-averaged spectral measures, namely, the density of states measure.

Let us recall the definition of the density of states measure. By the spectral theorem, there are Borel probability measures $d\mu_{V,\omega}$ on $\R$ such that
$$
\langle \delta_0 , g(H_{V,\omega}) \delta_0 \rangle = \int g(E) \, d\mu_{V,\omega}(E)
$$
for all bounded measurable functions $g$. The \emph{density of states measure} $dN_V$ is given by the $\omega$-average of these measures with respect to Lebesgue measure, that is,
$$
\int_\T \langle \delta_0 , g(H_{V,\omega}) \delta_0 \rangle \, d\omega= \int g(E) \, dN_V(E)
$$
for all bounded measurable functions $g$. By general principles, the density of states measure is non-atomic and its topological support is $\Sigma_V$.

The density of states measure can also be obtained by counting the number of eigenvalues per unit volume, in a given energy region, of restrictions of the operator to finite intervals (which explains the terminology). Indeed, for any real $a < b$,
$$
N_V(a,b) = \lim_{L \to \infty} \frac{1}{L} \# \big\{ \text{eigenvalues of } H_{V,\omega}|_{[1,L]} \text{ that lie in } (a,b) \big\},
$$
uniformly in $\omega$. Here, for definiteness, $H_{V,\omega}|_{[1,L]}$ is defined with Dirichlet boundary conditions. Note that the definitions of $H_{V,\omega}$, $\Sigma_V$, and $dN_V$ extend naturally to $V = 0$. It is well known that $\Sigma_0 = [-2,2]$ and
\begin{equation}\label{e.freeids}
N_0(-\infty,E) = \begin{cases} 0 & E \le -2 \\ \frac{1}{\pi} \arccos \left( - \frac{E}{2} \right) & -2 < E < 2 \\ 1 & E \ge 2. \end{cases}
\end{equation}

The following theorems are the main results of this paper:

\begin{Thm}\label{t.main}
There exists $0 < V_0 \le \infty$ such that for $V \in (0,V_0)$, there is $d_V \in (0,1)$ so that the density of states measure $dN_V$ is of exact dimension $d_V$, that is, for $dN_V$-almost every $E \in \R$, we have
$$
\lim_{\varepsilon \downarrow 0} \frac{\log N_V(E - \varepsilon , E + \varepsilon)}{\log \varepsilon} = d_V.
$$
Moreover, in $(0,V_0)$,  $d_V$ is a $C^\infty$ function of $V$, and
$$
\lim_{V \downarrow 0} d_V = 1.
$$
\end{Thm}

\begin{Thm}\label{t.main2}
For $V>0$ sufficiently small, we have $d_V < \dim_H \Sigma_V$.
\end{Thm}

\noindent\textit{Remarks.}
(a) Theorem~\ref{t.main2} confirms a conjecture of Barry Simon (that $\inf \{ \dim_H(S) : N_V(S) = 1 \} < \dim_H \Sigma_V$ for all $V > 0$), at least in the small coupling regime. The conjecture is motivated by the fact that the density of states measure is the equilibrium measure on the spectrum and that results of the same kind are known for equilibrium measures on other kinds of Cantor sets; compare Remark~(g) below.
\\[1mm]
(b) To the best of our knowledge, this provides the first example of an ergodic family of Schr\"odinger operators with singular density of states measure for which exact dimensionality can be shown. In the more general class of Jacobi matrices, one can use inverse spectral theory to generate examples with this property (this is implicit in \cite{BBM}), but these examples are not of Schr\"odinger form.  %\footnote{In the more general class of Jacobi matrices, one can use inverse spectral theory to generate examples with this property; this is implicit in \cite{BBM}. These examples are not of Schr\"odinger form.}
\\[1mm]
(c) After this work was completed, we learned from Serge Cantat that he has closely related (unpublished) work.
\\[1mm]
(d) Similar results can be shown for the density of states measure associated with potentials generated by invertible primitive substitutions over two symbols. We do not elaborate on this here and leave the details to a future publication.
\\[1mm]
(e) The value of $V_0$ corresponds to the possible ``breakdown of transversality.'' As we will recall in Section~\ref{s.2} below, a crucial role in the analysis of the Fibonacci Hamiltonian is played by the dynamical properties of the trace map, a polynomial transformation of $\R^3$, which may be restricted to invariant surfaces $S_V$ indexed by the coupling constant $V$. It is known that for each of the restrictions, the non-wandering set is hyperbolic, and hence we obtain foliations into stable and unstable manifolds. The energy parameter of the spectral problem corresponds to the points of a line of initial conditions lying in the invariant surface. The statements in Theorem~\ref{t.main} hold for any $V_0$ for which this line is transversal to the stable manifolds of the points in the non-wandering set. It is known that this transversality condition holds for $V \in (0,V_0)$ if $V_0 > 0$ is sufficiently small. On the other hand, it is well possible that transversality holds for all $V > 0$ (and hence $V_0$ may be chosen to be infinite). In any event, the present paper motivates a further investigation of the transversality issue with the goal of identifying the optimal choice of $V_0$.
\\[1mm]
(f) Theorem~\ref{t.main} implies that for $V \in (0,V_0)$,
\begin{align*}
d_V & = \inf \{ \dim_H(S) : N_V(S) = 1 \} \\
& = \inf \{ \dim_H(S) : N_V(S) > 0 \} \\
& = \inf \{ \dim_P(S) : N_V(S) = 1 \} \\
& = \inf \{ \dim_P(S) : N_V(S) > 0 \}.
\end{align*}
Here, $\dim_H(S)$ and $\dim_P(S)$ denote the Hausdorff dimension and the packing dimension of $S$, respectively; see Falconer \cite[Sections~2.1 and 10.1]{F} for definitions and the description of the four numbers above in terms of local scaling exponents of the measure and their coincidence in case the local scaling exponents are almost everywhere constant. Since the spectrum $\Sigma_V$ is the topological support of $N_V$, its dimension bounds $d_V$ from above.
\\[1mm]
(g) The density of states measure can be interpreted in a variety of ways in addition to the spectral definition given above. Via the Thouless formula, it can be seen that $dN_V$ is the equilibrium measure for the set $\Sigma_V$ (in the sense of logarithmic potential theory) since the Lyapunov exponent of the associated Schr\"odinger cocycles vanishes throughout the spectrum; see Simon \cite[Theorem~1.15]{S07} for this consequence of vanishing Lyapunov exponents in the context of ergodic Jacobi matrices. This in turn shows that $dN_V$ is also equal to the harmonic measure on $\Sigma_V$, relative to the point $\infty$; compare \cite{GM}. As a consequence, the statement that $\inf \{ \dim_H(S) : N_V(S) = 1 \}$ is strictly less than $\dim_H \Sigma_V$ is related in spirit to work of Makarov and Volberg \cite{Mak, MV, V}.
\\[1mm]
(h) We expect that $\lim_{V \downarrow 0} \frac{1 - d_V}{V}$ exists. In fact, since it is known \cite{DG3} that there are positive constants $C_1 , C_2$ such that for $V > 0$ small enough, $C_1 V \le 1 - \dim_H \Sigma_V \le C_2 V$, it follows that, if $\lim_{V \downarrow 0} \frac{1 - d_V}{V}$ exists, it must be strictly positive.
\\[1mm]
(i) Taking into account Remark (c), we expect that $d_V$ is well defined for all $V>0$. Moreover, it is reasonable to expect that $d_V$ is an analytic function of $V>0$, and $d_V< \dim_H \Sigma_V$ for all $V>0$, but we have not been able to prove this.
\\[1mm]
(j) The fact that $d_V$ is close to $1$ for small $V > 0$ potentially has applications to phase-averaged transport.  It is well known that dimensional estimates for spectral measures yield lower bounds for moments of the position operator. As alluded to above, the known dimensional estimates for the spectral measures of $H_{V,\omega}$ are relatively weak, and are in particular not close to $1$. The fact that phase-averaged spectral measures have dimension close to $1$ for $V$ small give rise to the hope that, as a consequence, phase-averaged moments of the position operator satisfy a correspondingly strong lower bound. Let us make explicit what needs to be shown to derive such a conclusion. Suppose we have the initial state $\psi(0)=\delta_0$ that is exposed to a time-evolution given by the Schr\"odinger equation $i \partial_t \psi(t) = H_{V,\omega} \psi(t)$. The solution is given by $\psi(t) = e^{-itH_{V,\omega}} \delta_0$. Thus, $n \mapsto | \langle \delta_n , \psi(t) \rangle |^2$ is a probability distribution on $\Z$ (describing the probability that the state is at site $n$ at time $t$). For $p > 0$, the $p$-th moment of this distribution is given by
$$
M(p,t,V,\omega) = \sum_{n \in \Z} |n|^p | \langle \delta_n , e^{-itH_{V,\omega}} \delta_0 \rangle |^2.
$$
Let us now average with respect to $t$ and $\omega$ and consider, for $T > 0$,
$$
\langle M \rangle (p,T,V) = \int_\T \int_0^T M(p,t,V,\omega) \, \frac{dt}{T} \, d\omega.
$$
To capture the power-law growth (in $T$) of these quantities, we define the associated transport exponents by
$$
\beta^-(p,V) = \liminf_{T \to \infty} \frac{\log (1 + \langle M \rangle (p,T,V))}{p \log T}.
$$
It follows from general principles that $\beta^-(p,V)$ is non-decreasing in $p$ and takes values in $[0,1]$. One would expect that $\beta^-(\cdot,V)$ is uniformly close to $1$ if $V$ is close to $0$. We thus ask the question whether
\begin{equation}\label{e.conj}
\inf_{p > 0} \beta^-(p,V) \ge d_V
\end{equation}
holds for the Fibonacci Hamiltonian. Let us note that a result of this kind has been shown for (generalizations of) the critical almost Mathieu operator by Bellissard, Guarneri, and Schulz-Baldes \cite{BGS}. We also note that a result like \eqref{e.conj} cannot hold in complete generality as random or pseudo-random potentials provide counterexamples.

%\bigskip

%The structure of the paper is the following. We first describe the measure of maximal entropy for $T_V |_{\Omega_V}$. Then we describe the density of states measure as a Markov measure on a piece of an unstable manifold of the trace map using the previous construction. Then we deduce that the density of states measure is exact dimensional using the dimension theory of dynamical systems. Finally, we use the relation between entropy, Lyapunov exponent, and the dimension of the measure in order to show that $d_V \to 1$ as $V \to 0$. This relation also shows how the explicit estimates from Remark~(b) above could be obtained. We will only sketch the details.

%\section{The Measure of Maximal Entropy for $T_V$}

%\section{The Density of States Measure as a Markov Measure}

%\section{Relation Between Entropy, Lyapunov Exponent, and the Dimension of the Density of States Measure}

%\section{Properties of Lyapunov Exponent}

\section{Preliminaries}\label{s.2}

In this section we summarize some known results and connections which we will use in the proof of Theorems \ref{t.main} and \ref{t.main2} given in the next section.

\subsection{The Trace Map}

There is a fundamental connection between the spectral properties of the Fibonacci Hamiltonian and the dynamics of the \textit{trace map}
$$
T : \Bbb{R}^3 \to \Bbb{R}^3, \; T(x,y,z)=(2xy-z,x,y).
$$
The function
\beq\label{e.FVinvariant}
G(x,y,z) = x^2+y^2+z^2-2xyz-1
\eneq
is invariant under the action of $T$, and hence $T$ preserves the family of cubic surfaces
$$
S_V = \left\{(x,y,z)\in \Bbb{R}^3 : x^2+y^2+z^2-2xyz=1+ \frac{V^2}{4} \right\}.
$$
It is therefore natural to consider the restriction $T_{V}$ of the trace map $T$ to the invariant surface $S_V$. That is, $T_{V}:S_V \to S_V$, $T_{V}=T|_{S_V}$. We denote by $\Omega_{V}$ the set of points in $S_V$ whose full orbits under $T_{V}$ are bounded. It is known that $\Omega_V$ is equal to the non-wandering set of $T_V$; indeed, it follows from \cite{Ro} that every unbounded orbit must escape to infinity together with its neighborhood (either in positive or negative time), hence is wondering, and hyperbolicity of $\Omega_V$ shown in \cite{Can} implies that every point of $\Omega_V$ is non-wandering.

\subsection{Hyperbolicity of the Trace Map}

Recall that an invariant closed set $\Lambda$ of a diffeomorphism $f : M \to M$ is \textit{hyperbolic} if there exists a splitting of the tangent space $T_xM=E^u_x\oplus E^u_x$ at every point $x\in \Lambda$ such that this splitting is invariant under $Df$, the differential $Df$ exponentially contracts vectors from the stable subspaces $\{E^s_x\}$, and the differential of the inverse, $Df^{-1}$, exponentially contracts vectors from the unstable subspaces $\{E^u_x\}$. A hyperbolic set $\Lambda$ of a diffeomorphism $f : M \to M$ is \textit{locally maximal} if there
exists a neighborhood $U$ of $\Lambda$ such that
$$
\Lambda=\bigcap_{n\in\Bbb{Z}}f^n(U).
$$
It is known that for $V > 0$, $\Omega_{V}$ is a locally maximal compact transitive hyperbolic set of $T_{V} : S_V \to S_V$; see \cite{Can, Cas, DG1}.

\subsection{Properties of the Trace Map for $V=0$}\label{ss.vequzero}

The surface
$$
\mathbb{S} = S_0 \cap \{ (x,y,z)\in \Bbb{R}^3 : |x|\le 1, |y|\le 1, |z|\le 1\}
$$
is homeomorphic to $S^2$, invariant under $T$, smooth everywhere except at the four points $P_1=(1,1,1)$, $P_2=(-1,-1,1)$, $P_3=(1,-1,-1)$, and $P_4=(-1,1,-1)$, where $\mathbb{S}$ has conic singularities, and the trace map $T$ restricted to $\mathbb{S}$ is a factor of the hyperbolic automorphism of $\T^2 = \R^2 / \Z^2$ given by
\beq\label{e.a}
\mathcal{A}(\theta, \varphi) = (\theta + \varphi, \theta)\ (\text{\rm mod}\ 1).
\eneq
The semi-conjugacy is given by the map
\begin{equation}\label{e.semiconj}
F: (\theta, \varphi) \mapsto (\cos 2\pi(\theta + \varphi), \cos 2\pi \theta, \cos 2\pi \varphi).
\end{equation}
The map $\mathcal{A}$ is hyperbolic, and is given by the matrix $A = \begin{pmatrix} 1 & 1 \\ 1 & 0 \end{pmatrix}$, which has eigenvalues
$$
\mu=\frac{1+\sqrt{5}}{2}\ \ \text{\rm and} \ \ \ -\mu^{-1}=\frac{1-\sqrt{5}}{2}.
$$
A Markov partition for the map $\mathcal{A}:\mathbb{T}^2\to \mathbb{T}^2$ is shown in
Figure~\ref{fig.Casdagli-Markov1}. Its image under the map $F:\mathbb{T}^2\to \mathbb{S}$ is a Markov
partition for the pseudo-Anosov map $T:\mathbb{S}\to \mathbb{S}$.

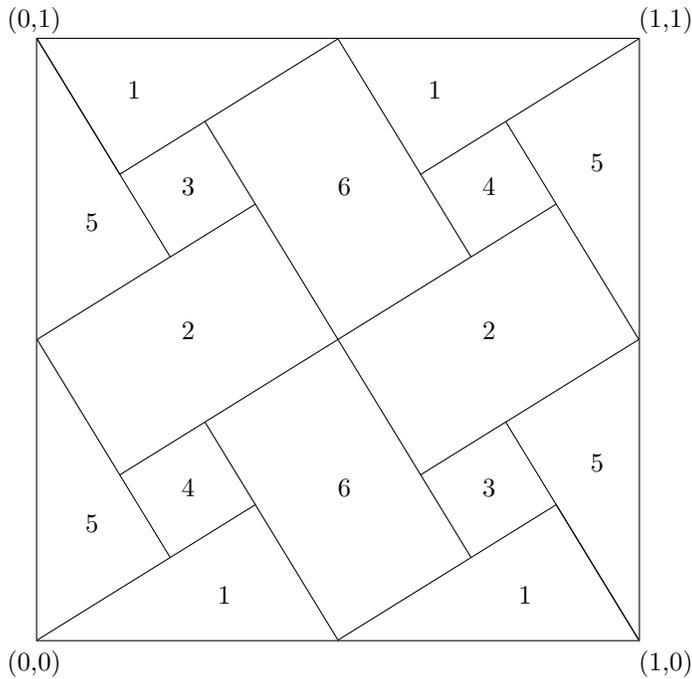
\begin{figure}[htb]
\setlength{\unitlength}{0.8mm}
\begin{picture}(120,120)

\put(10,10){\framebox(100,100)}

\put(5,5){(0,0)}

\put(110,5){(1,0)}

\put(5,112){(0,1)}

\put(110,112){(1,1)}

\put(10,10){\line(161,100){36.2}}

\put(60,10){\line(-61,100){13.8}}

\put(40,16){$1$}

\put(60,10){\line(161,100){36.2}}

\put(110,10){\line(-61,100){13.8}}

\put(90,16){$1$}

\put(60,110){\line(-161,-100){36.2}}

\put(10,110){\line(61,-100){13.8}}

\put(25,100){$1$}

\put(110,110){\line(-161,-100){36.2}}

\put(60,110){\line(61,-100){13.8}}

\put(75,100){$1$}

\put(10,60){\line(61,-100){22.1}}

\put(18,28){$5$}

\put(10,60){\line(161,100){22.1}}

\put(10,110){\line(61,-100){22.1}}

\put(18,78){$5$}

\put(110,60){\line(-61,100){22.1}}

\put(102,38){$5$}

\put(110,60){\line(-161,-100){22.1}}

\put(110,10){\line(-61,100){22.1}}

\put(102,88){$5$}

\put(60,60){\line(-161,-100){36.3}}

\put(60,60){\line(161,100){36.3}}

\put(60,60){\line(-61,100){22.1}}

\put(60,60){\line(61,-100){22.1}}

\put(88,46.35){\line(-161,-100){14.3}}

\put(32,73.65){\line(161,100){14.3}}

\put(37.88,46.28){\line(61,-100){8.4}}

\put(82.12,73.72){\line(-61,100){8.4}}

\put(34,34){$4$}

\put(84,84){$4$}

\put(34,84){$3$}

\put(84,34){$3$}

\put(34,60){$2$}

\put(84,60){$2$}

\put(60,34){$6$}

\put(60,84){$6$}

\end{picture}
\caption{The Markov partition for the map $\mathcal{A}$.}\label{fig.Casdagli-Markov1}
\end{figure}

\subsection{Spectrum and Trace Map}

Denote by $\ell_V$ the line
$$
\ell_V = \left\{ \left(\frac{E-V}{2}, \frac{E}{2}, 1 \right) : E \in \Bbb{R} \right\}.
$$
It is easy to check that $\ell_V \subset S_V$. An energy $E \in \Bbb{R}$ belongs to the spectrum $\Sigma_V$ of the Fibonacci Hamiltonian if and only if the positive semiorbit of the point $(\frac{E-V}{2}, \frac{E}{2}, 1)$ under iterates of the trace map $T$ is bounded; see \cite{S87}. Moreover, stable manifolds of points in $\Omega_V$ intersect the line $\ell_V$ transversally if $V > 0$ is sufficiently small \cite{DG1} or if $V \ge 16$ \cite{Cas}. It is an open problem whether this transversality condition holds for all $V > 0$. The critical value $V_0$ in Theorem~\ref{t.main} is given by the largest number in $(0,\infty]$ for which the transversality condition holds for $V \in (0,V_0)$.

\section{Proof of Theorems \ref{t.main} and \ref{t.main2}}

This section is devoted to the proof of Theorems \ref{t.main} and \ref{t.main2}. In the proof we use freely notions and notation that are standard in the modern theory of dynamical systems; the comprehensive encyclopedia \cite{KH} can be used as a standard reference.    %We first prove exact-dimensionality and then address $d_V < \dim_H \Sigma_V$, smoothness of $d_V$ in $V$, and $\lim_{V \downarrow 0} d_V = 1$.

Consider the semi-conjugacy $F : \T^2 \to \mathbb{S}$, compare Figure~\ref{fig.semiconj}, and the push-forward $\mu$ of Lebesgue measure under $F$, that is, $\mu_0 = F_*(\mathrm{Leb})$.

\begin{figure}[htb] {\begin{minipage}{5cm} \includegraphics[width=1.2\textwidth]{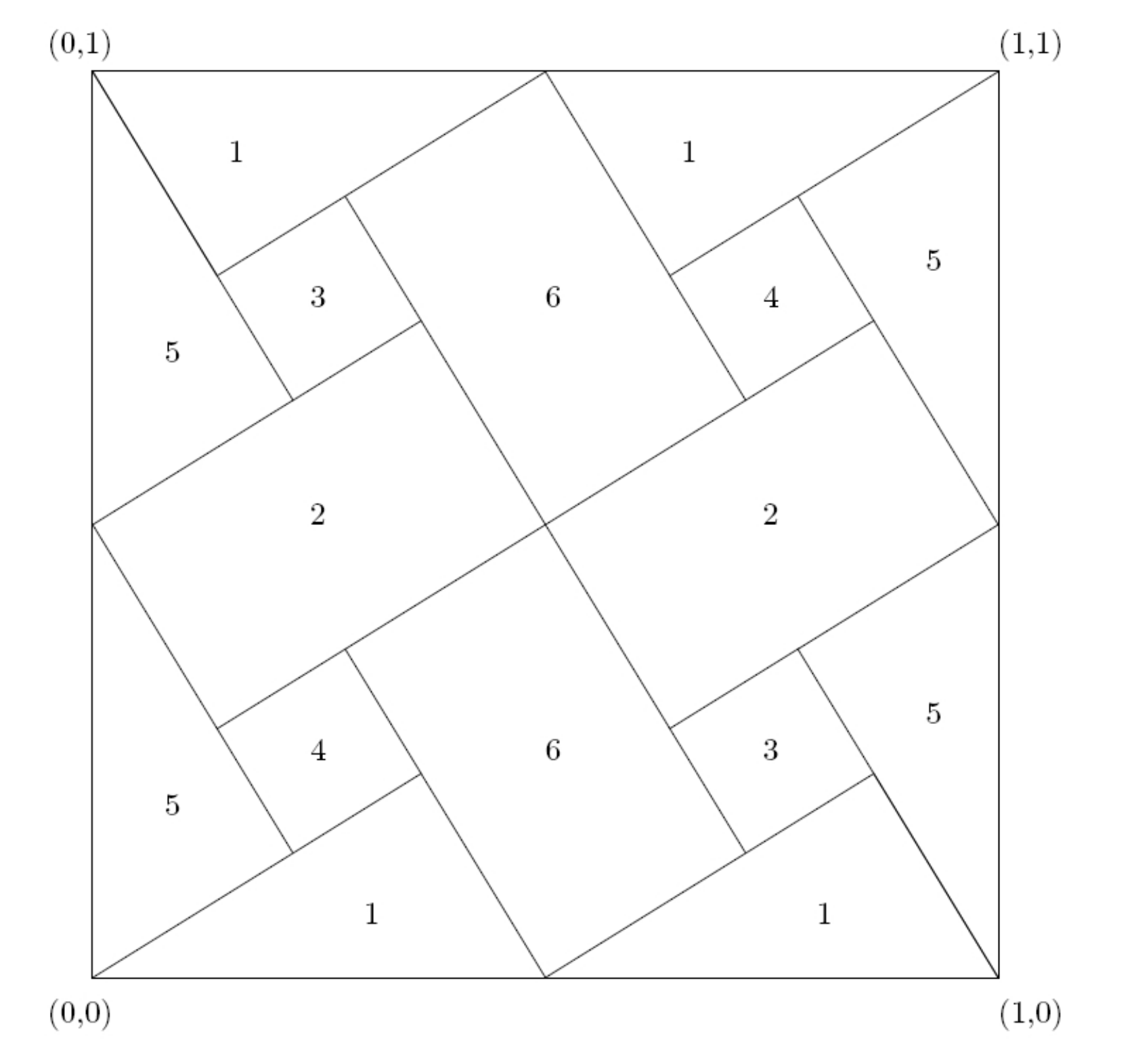}
\label{fig:sing0.0} \end{minipage} \; \; \; \; \; {\Large $\xrightarrow{F}$} \ \
\begin{minipage}{5cm}
\includegraphics[width=1.25\textwidth]{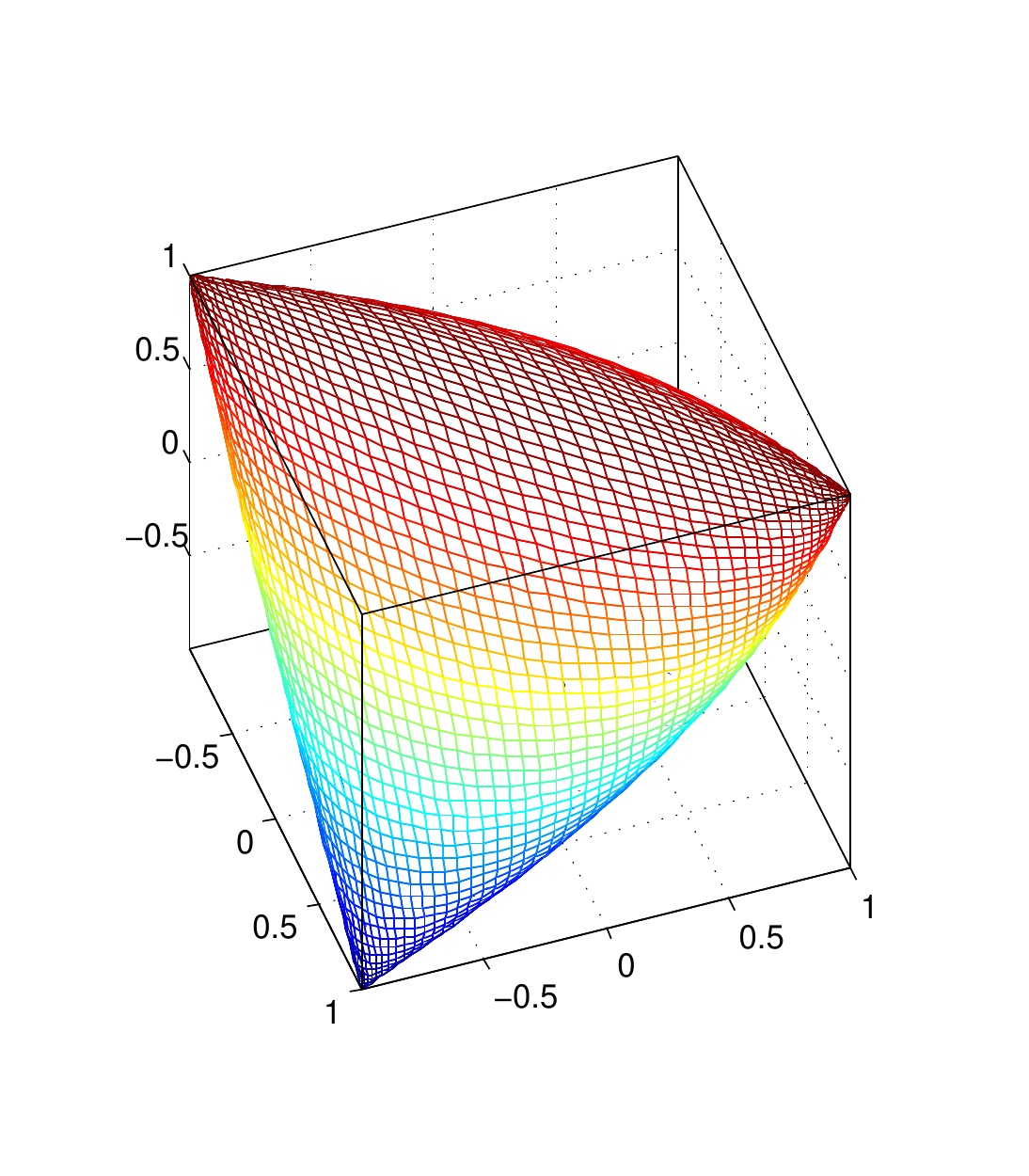}
\end{minipage}}
\caption{The semi-conjugacy $F$ between the linear map $\mathcal{A}$ and the trace map $T$ on the central part $\mathbb{S}$ of the Cayley cubic.}\label{fig.semiconj}
\end{figure}

Since $\mathcal{A} : \T^2 \to \T^2$ is linear, $\mathrm{Leb}$ on $\T^2$ is the measure of maximal entropy for $\mathcal{A}$. We have
$$
h_\mathrm{top}(T|_\mathbb{S}) = h_\mathrm{top}(\mathcal{A}) = h_\mathrm{Leb}(\mathcal{A}) = h_{\mu_0}(T) = \log \frac{3 + \sqrt{5}}{2},
$$
and hence $\mu_0$ is the measure of maximal entropy for $T|_\mathbb{S}$.

Consider the Markov partition for $\mathcal{A}$, described in Subsection~\ref{ss.vequzero} and shown in Figure~\ref{fig.Casdagli-Markov1}, and the element $R$ that contains the interval $I = [0,1/2] \times \{0\}$. Note that the measure on this interval obtained by projection of normalized Lebesgue measure on $R$ along the stable manifolds coincides with normalized Lebesgue measure on $I$.

\begin{Cl}\label{c.1}
The push-forward of the normalized Lebesgue measure on $I$ under $F$, which is a probability measure on $\ell_0 \cap \mathbb{S}$, corresponds to the free density of states measure under the identification
\begin{equation}\label{e.ident}
J_V: E \mapsto \left( \frac{E - V}{2} , \frac{E}{2} , 1 \right)
\end{equation}
{\rm (}here with $V = 0$, and as an identification of $\Sigma_0 = [-2,2]$ and $\ell_0 \cap \mathbb{S}${\rm )}.
\end{Cl}

\begin{proof}
This follows from the explicit of form of the semi-conjugacy $F$ and the free density of states measure $dN_0$ on the free spectrum $\Sigma_0 = [-2,2]$; see \eqref{e.semiconj} and \eqref{e.freeids}.
\end{proof}

\begin{Cl}\label{c.2}
There exists $0<V_0\le \infty$ such that for $V\in (0, V_0)$, the following statement holds: Consider the line $\ell_V \subset S_V$ and the continuation $\mathcal{R}_V$ of the element $\mathcal{R}_0$ of the Markov partition of $\mathbb{S}$ chosen above for zero coupling that contains $\ell_0 \cap \mathbb{S}$ {\rm (}namely, the image of $R$ under $F${\rm )} to the given value of $V$. Restrict the measure of maximal entropy for $T|_{\Omega_V}$ to $\mathcal{R}_V$, normalize it, and project it to $\ell_V$ along stable leaves. The resulting probability measure $d\tilde N_V$ on $\ell_V$ corresponds to $dN_V$ under the identification \eqref{e.ident}.
\end{Cl}

\begin{proof}
Let us take $V_0$ as described above in Remark (e). Every gap $U_V$ in the spectrum $\Sigma_V$ has a continuous continuation for all $V \in (0, V_0)$; see Theorem 1.5 from \cite{DG3}.

Suppose we have a family of measures $\{\nu_{V}\}$ on $\mathbb{R}$, $V\in [0, V_0)$, satisfying the following properties:

\vspace{8pt}

(i) $\text{supp}\,\nu_V=\Sigma_V$;

\vspace{5pt}

(ii) $\nu_0=dN_0$;

\vspace{5pt}

(iii) $\{\nu_{V}\}$ depends continuously on $V$, $V\in [0, V_0)$;

\vspace{5pt}

(iv) for any two continuous families of gaps $\{U_V\}$ and $\{W_V\}$ in the spectrum $\Sigma_V$, the measure $\nu_V(E_1, E_2)$, where $E_1 \in U_V$ and $E_2 \in W_V$, is independent of $V$.

\vspace{5pt}

Notice that since $\Sigma_V$ for $V\in (0, V_0)$ is a Cantor set, the properties (i)--(iv) imply that $\nu_V = dN_V$ for $V\in[0, V_0)$. We will show that the family of measures $\left(J_V^{-1}\right)_*(d\widetilde{N}_V)$, where $d\widetilde{N}_V$ is a measure on $\ell_V$ described in the statement of Claim~\ref{c.2}, satisfies the properties (i)--(iv). This will prove Claim \ref{c.2}.

Property (i) holds for $\left(J_V^{-1}\right)_*(d\widetilde{N}_V)$ since $\Sigma_V = J_V^{-1} \left( W^s(\Omega_V)\cap \ell_V \right)$. Indeed, an energy $E \in \Sigma_V$  if and only if the positive semiorbit of the point $J_V(E)$ under iterates of $T$ is bounded \cite{S87}. On the other hand, the set of all points with bounded positive semiorbits is exactly the stable lamination $W^s(\Omega_V)$ of the hyperbolic set $\Omega_V$; see \cite{Can}.

Property (ii) follows from the fact that Lebesgue measure is a measure of maximal entropy for the hyperbolic automorphism $\mathcal{A}$ \eqref{e.a} and Claim \ref{c.1}.

In order to check properties (iii) and (iv), we need to present a better description of the measure of maximal entropy for $T|_{\Omega_V}$. We know that the restriction $T|_{\Omega_V}$ is topologically conjugate to a topological Markov chain $\sigma:\Sigma_{TMC}\to \Sigma_{TMC}$; see \cite{DG1} for the explicit description of the topological Markov chain $\Sigma_{TMC}$. Denote the conjugacy by $h_V:\Sigma_{TMC}\to \Omega_V$. Let $\mu_{max}$ be the measure of maximal entropy for $\sigma:\Sigma_{TMC}\to \Sigma_{TMC}$, and set $\mu_V=(h_V)_*\mu_{max}$. As $V \to 0$, the family of conjugacies $\{h_V\}$ converges to a semiconjugacy $h_0:\Sigma_{TMC}\to \mathbb{S}$ between the topological Markov chain $\sigma:\Sigma_{TMC}\to \Sigma_{TMC}$ and $T|_{\mathbb{S}}$. Notice that $(h_0)_*\mu_{max}=F_*(Leb)$ since both measures are measures of maximal entropy for $T|_{\mathbb{S}}$. Since the family of maps $\{h_V\}_{V\in[0, V_0)}$ as well as the family of projections along the stable leaves depend continuously on $V$, property~(iii) follows.

Finally, consider Figure~\ref{fig.Casdagli-Markov}, which shows the Markov partition for the map $\mathcal{A}$, the rectangle $R$, and the partition of $R$ by pieces of stable manifolds. Consider the push-forward of this picture under $F$. This gives a partition of $F(R) = \mathcal{R}_0$ into small rectangles. As $V$  becomes strictly positive, the gaps in $\Sigma_V$ open precisely at the points in $\Sigma_0 = [-2,2]$ where the the stable manifolds of singularities intersect $\ell_0$ (modulo the identification \eqref{e.ident}). For given families of gaps $\{U_V\}$ and $\{W_V\}$ in the spectrum $\Sigma_V$, consider the families of gaps $\{J_V(U_V)\}$ and $\{J_V(W_V)\}$ on $\ell_V$. The boundaries of these gaps are on stable manifolds of periodic points that were born from singularities. The measure $d\widetilde{N}_V$ of the interval on $\ell_V$ between these gaps is given by the (normalized) measure $\mu_V$ of a sub-rectangle $\Pi_V \subset \mathcal{R}_V$ formed by these stable manifolds. We have $\mu_V(\Pi_V) = \mu_{max}\left(h^{-1}_V(\Pi_V)\right)$. But the set $h^{-1}_V(\Pi_V)\subset \Sigma_{TMC}$ is $V$-independent, therefore the measure $d\widetilde{N}_V$ of the interval between $\{J_V(U_V)\}$ and $\{J_V(W_V)\}$ is also $V$-independent. This proves property (iv) and hence Claim \ref{c.2}.
\end{proof}
%
%
%Consider Figure~\ref{fig.Casdagli-Markov}, which shows the Markov partition for the map $\mathcal{A}$, the rectangle $R$, and the partition of $R$ by pieces of stable manifolds. Consider the push-forward of this picture under $F$. This gives a partition of $F(R) = \mathcal{R}_0$ into small rectangles and, by Claim~\ref{c.1}, we can describe the measure $dN_0$ on $\ell_0 \cap \mathbb{S}$ (modulo the identification \eqref{e.ident}) in terms of the measure of these rectangles with respect to $F_*(\mathrm{Leb}) = \mu_0$.

\begin{figure}[htb]
\setlength{\unitlength}{0.8mm}
\begin{picture}(120,120)

\put(10,10){\framebox(100,100)}

\put(5,5){(0,0)}

\put(110,5){(1,0)}

\put(5,112){(0,1)}

\put(110,112){(1,1)}

\put(10,10){\line(161,100){36.2}}

\put(60,10){\line(-61,100){13.8}}

%\put(40,16){$1$}

%

%

\put(60,10){\line(161,100){36.2}}

\put(110,10){\line(-61,100){13.8}}

\put(90,16){$1$}

\put(60,110){\line(-161,-100){36.2}}

\put(10,110){\line(61,-100){13.8}}

\put(25,100){$1$}

\put(110,110){\line(-161,-100){36.2}}

\put(60,110){\line(61,-100){13.8}}

\put(75,100){$1$}

\put(10,60){\line(61,-100){22.1}}

\put(18,28){$5$}

\put(10,60){\line(161,100){22.1}}

\put(10,110){\line(61,-100){22.1}}

\put(18,78){$5$}

\put(110,60){\line(-61,100){22.1}}

\put(102,38){$5$}

\put(110,60){\line(-161,-100){22.1}}

\put(110,10){\line(-61,100){22.1}}

\put(102,88){$5$}

\put(60,60){\line(-161,-100){36.3}}

\put(60,60){\line(161,100){36.3}}

\put(60,60){\line(-61,100){22.1}}

\put(60,60){\line(61,-100){22.1}}

\put(88,46.35){\line(-161,-100){14.3}}

\put(32,73.65){\line(161,100){14.3}}

\put(37.88,46.28){\line(61,-100){8.4}}

\put(82.12,73.72){\line(-61,100){8.4}}

\put(34,34){$4$}

\put(84,84){$4$}

\put(34,84){$3$}

\put(84,34){$3$}

\put(34,60){$2$}

\put(84,60){$2$}

\put(60,34){$6$}

\put(60,84){$6$}

\put(33,9){$R$}

\put(10,10){\line(61,-100){13.8}}

\put(60,10){\line(-161,-100){36.2}}

\put(60,5){$\left(\frac{1}{2}, 0\right)$}

\put(13.415, 12.121){\line(61,-100){13.8}}

\put(16.83, 14.242){\line(61,-100){13.8}}

\put(20.245, 16.363){\line(61,-100){13.8}}

\put(23.66, 18.484){\line(61,-100){13.8}}

\put( 27.075,  20.605){\line(61,-100){13.8}}

\put( 30.49,  22.726){\line(61,-100){13.8}}

\put(33.905,  24.847){\line(61,-100){13.8}}

\put(37.32,  26.968){\line(61,-100){13.8}}

\put( 40.737,  29.089){\line(61,-100){13.8}}

\put( 44.15,  31.21){\line(61,-100){13.8}}

\put( 11.138,  10.707){\line(61,-100){13.8}}

\put( 14.553,  12.828){\line(61,-100){13.8}}

\end{picture}

\

\

\

\caption{The Markov partition for the map $\mathcal{A}$, the rectangle $R$, and the partition of $R$ by pieces of stable manifolds.}\label{fig.Casdagli-Markov}
\end{figure}
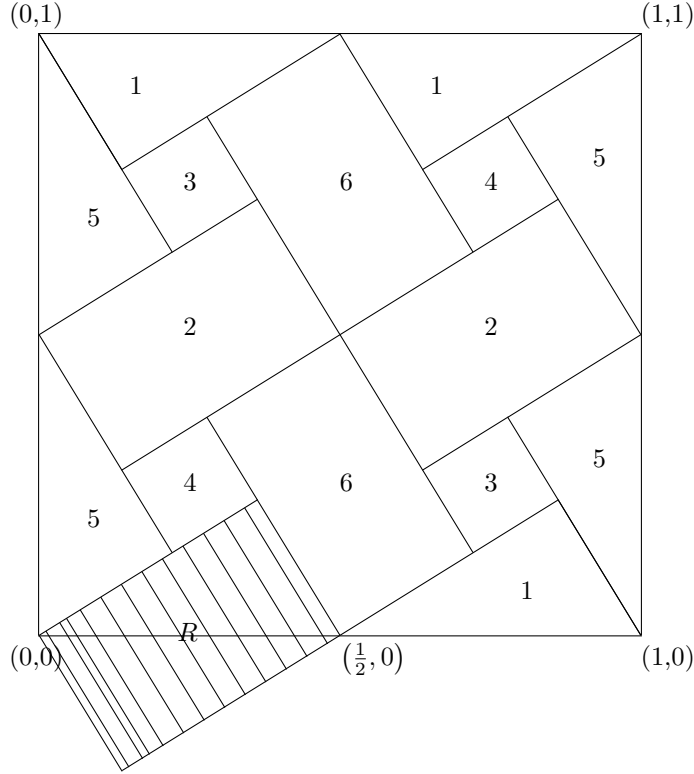

We are now in a position to prove the four main assertions of Theorems \ref{t.main} and \ref{t.main2}:

\begin{proof}[Proof {\rm (}that $dN_V$ is exact-dimensional{\rm )}.]
Consider $\mu_V$, the measure of maximal entropy for $T|_{\Omega_V}$. It has a local product structure and hence
\begin{equation}\label{e.lps}
\frac{1}{\mu_V(\Pi_V)} \mu_V|_{\Pi_V} \simeq \mu_V^s \times \mu_V^u.
\end{equation}
Due to \cite[Proposition 26.1]{P}, $\mu_V^u$ is exact-dimensional. Denote by $d_V$ the almost everywhere constant value of the local scaling exponent of $\mu_V^u$. We note for later use that \cite[Proposition 26.1]{P} also implies that
\begin{equation}\label{e.pform}
d_V = \frac{h_\mathrm{top}(T|_{\Omega_V})}{\mathrm{Lyap}^u(\mu_V)},
\end{equation}
where $\mathrm{Lyap}^u(\mu_V)$ is a largest Lyapunov exponent with respect to $\mu_V$.
By Claim~\ref{c.2} and \eqref{e.lps}, the image of $\mu_V^u$ under the stable holonomy map is (through \eqref{e.ident} equivalent to) the density of states measure $dN_V$. Moreover, since the stable foliation is $C^1$, exact-dimensionality and local scaling exponents are preserved, so that the exact-dimensionality of $dN_V$ with almost sure local scaling exponent $d_V$ follows.
\end{proof}

\begin{proof}[Proof {\rm (}that $d_V < \dim_H \Sigma_V$ for small $V > 0${\rm )}.]
For $p \in \Omega_V$, let
$$
\varphi_V(p) = - \log \left\| DT_p|_{E^u_p} \right\|,
$$
and let $\nu_V$ be the equilibrium measure for the potential $(\dim_H \Sigma_V) \cdot \varphi_V$. Then, on $W^u_\mathrm{loc}$, the set of regular points for $\nu_V$ has Hausdorff dimension $\dim_H \Sigma_V$ due to \cite[Theorem~1]{MM}. Moreover, the proof of \cite[Corollary~3]{MM} implies that
\begin{equation}\label{e.mmform}
\frac{h_\mathrm{top}(T|_{\Omega_V})}{\mathrm{Lyap}^u(\mu_V)} \le \dim_H \Sigma_V,
\end{equation}
and equality holds in \eqref{e.mmform} if and only if
$$
\sup_\xi \left( h_\xi - \dim_H \Sigma_V \cdot \mathrm{Lyap}^u (\xi) \right) = h_\mathrm{top} - \dim_H \Sigma_V \cdot \mathrm{Lyap}^u(\mu_V).
$$
Due to uniqueness of equilibrium states, it follows that equality holds in \eqref{e.mmform} if and only if $\mu_V$ is the equilibrium state for the potential $\dim_H \Sigma_V \cdot \varphi$. The latter condition implies by \cite[Proposition~20.3.10]{KH} that the potential $\dim_H \Sigma_V \cdot \varphi_V$ is cohomologous to a constant potential. This in turn implies by \cite[Theorem~19.2.1]{KH} that the averaged multipliers over all periodic points of $T|_{\Omega_V}$ must be the same.

We claim that the last statement fails for $V > 0$ sufficiently small. Putting everything together, it follows that we have strict inequality in \eqref{e.mmform} for $V > 0$ sufficiently small, which implies  $d_V < \dim_H \Sigma_V$ due to \eqref{e.pform}.

To conclude this part of the proof, let us prove the claim just made. For $a \in \R$, note that $T^6(0,0,a) = (0,0,a)$. On $S_V$, we find such a six-cycle for $a^2 = 1 + \frac{V^2}{4}$. We have
$$
DT^6(0,0,a) = \begin{pmatrix} 16 a^4 - 4 a^2 + 1 & 8 a^3 & 0 \\ 8 a^3 & 4 a^2 + 1 & 0 \\ 0 & 0 & 1 \end{pmatrix}.
$$
For $a = 1$ (i.e., $V = 0$), we have
$$
DT^6(0,0,1) = \begin{pmatrix} 18 & 8 & 0 \\ 8  & 5 & 0 \\ 0 & 0 & 1 \end{pmatrix},
$$
which has eigenvalues $\{ 1 , 9 - 4 \sqrt{5} , 9 + 4 \sqrt{5} \}$. Thus, for $V > 0$ small, the averaged multiplier for the six cycle (for $a$ appropriately chosen as indicated above), must be close to $(9 + 4 \sqrt{5})^{1/6}$.

On the other hand, the two-periodic points of $T$ are given by
$$
\mathrm{Per}_2(T) = \left\{ (x,y,z) : x = z , y = \frac{x}{2x-1} \right\}.
$$
For $V > 0$ small, the two-cycles in $\Omega_V \cap \mathrm{Per}_2(T)$ have multipliers near the squares of the eigenvalues of
$$
DT(1,1,1) =  \begin{pmatrix} 2 & 1 & 0 \\ 2  & 0 & 1 \\ -1 & 0 & 0 \end{pmatrix},
$$
and these eigenvalues are $\{ -1 , \frac{3 - \sqrt{5}}{2} , \frac{3 + \sqrt{5}}{2} \}$. Since $(9 + 4 \sqrt{5})^{1/6} \not= \frac{3 + \sqrt{5}}{2}$, the claim follows.
\end{proof}

\begin{proof}[Proof {\rm (}that $d_V$ is $C^\infty${\rm )}.]
Due to \eqref{e.pform},
$$
d_V = \frac{h_\mathrm{top}(T|_{\Omega_V})}{\mathrm{Lyap}^u(\mu_V)}.
$$
Here, $h_\mathrm{top}(T|_{\Omega_V})=\log\frac{3+\sqrt{5}}{2}$ is independent of $V$, and $\mathrm{Lyap}^u(\mu_V)=\int_{\Omega_V} (- \varphi_V) \, d\mu_V= \int_{\Sigma_{TMC}} (- \tilde \varphi_V(\omega)) \, d\tilde \mu(\omega)$, where $\tilde \mu$ is a measure of maximal entropy for $\sigma:\Sigma_{TMC}\to \Sigma_{TMC}$, and $\tilde \varphi_V=\varphi_V\circ h_V$. Therefore
$$
\frac{d}{dV}\mathrm{Lyap}^u(\mu_V)=-\int_{\Sigma_{TMC}} \left(\frac{d}{dV} \tilde \varphi_V(\omega)\right) \, d\tilde \mu(\omega).
$$
Notice that for a given $\omega\in \Sigma_{TMC}$, the set $\{h_V(\omega)\}_{V>0}$ forms a central leaf of the partially hyperbolic set $\cup_{V>0}\Omega$. The collection of these central leaves forms an invariant lamination which is $r$-normally hyperbolic for any $r\in \mathbb{N}$ (see \cite{HPS} for terminology and results on normal hyperbolicity). Indeed, since the central leaves are parameterized by the value of the Fricke-Vogt invariant $G$ (see \eqref{e.FVinvariant}), which is preserved by the trace map $T$, there is no asymptotic contraction or expansion along these central leaves. This implies (see Theorem 6.1 from \cite{HPS}) that the central leaves are $C^{\infty}$-curves, and their central-unstable invariant manifolds are also $C^{\infty}$. Therefore $\tilde \varphi_V(\omega)$ is a $C^{\infty}$ function for a fixed $\omega$, hence $\mathrm{Lyap}^u(\mu_V)$ and $d_V = \frac{h_\mathrm{top}(T|_{\Omega_V})}{\mathrm{Lyap}^u(\mu_V)}$ are also $C^{\infty}$ functions of $V$.
\end{proof}

\begin{proof}[Proof {\rm (}that $\lim_{V \downarrow 0} d_V = 1${\rm )}.]
Denote $L_V = \mathrm{Lyap}^u(\mu_V)$. Using \eqref{e.pform}, we may infer that
$$
d_V \cdot L_V = h_\mathrm{top}(T|_{\Omega_V}) = h_\mathrm{top}(T|_{\Omega_0}) = L_0.
$$
Thus, $d_V = \frac{L_0}{L_V}$ and it suffices to show that
\begin{equation}\label{e.lecont}
\lim_{V \downarrow 0} L_V = L_0.
\end{equation}

Lifting the measures $\mu_V$ to the topological Markov chain $\Sigma_\mathrm{TMC}$, we obtain the $V$-independent measure $\tilde \mu$. Let us also lift the potentials, that is, we let $\tilde \varphi_V = \varphi_V \circ h_V$. Then,
\begin{align*}
L_V & = \int_{\Omega_V} (- \varphi_V) \, d\mu_V \\
& = \int_{\Sigma_{TMC}} (- \tilde \varphi_V) \, d\tilde \mu \\
& \to \int_{\Sigma_{TMC}} (- \tilde \varphi_0) \, d\tilde \mu \\
& = L_0,
\end{align*}
since $\tilde \varphi_V \to \tilde \varphi_0$ pointwise (due to the continuity of $\{ E^u_V \}$) and all the potentials are uniformly bounded. This proves \eqref{e.lecont} and completes the proof of the theorem.
\end{proof}

\end{document}